\documentclass[reqno]{amsart}
\usepackage{amsmath}
\usepackage{amsthm}
\usepackage{mathtools}
\usepackage{url}
\usepackage{amssymb}
\usepackage{amssymb, amsmath}
\usepackage{amsfonts}
\usepackage[sort]{cite}

\newtheorem{theorem}{Theorem}[section]
\newtheorem{lemma}[theorem]{Lemma}
\newtheorem{corollary}[theorem]{Corollary}
\theoremstyle{remark}
\newtheorem{remark}[theorem]{Remark}

\newcommand{\relmiddle}[1]{\mathrel{}\middle#1\mathrel{}}

\newcommand{\RR}{\mathbb{R}}

\newcommand{\td}[2]{\tau^{\left( #1 \right)}_{ #2 }}

\newcommand{\ud}[2]{u^{\left( #1 \right)}_{ #2 }}
\newcommand{\vd}[2]{v^{\left( #1 \right)}_{ #2 }}
\newcommand{\ed}[2]{e^{\left( #1 \right)}_{ #2 }}

\newcommand{\tud}[2]{\tilde{u}^{\left( #1 \right)}_{ #2 }}
\newcommand{\fd}{\delta^+}
\newcommand{\bd}{\delta^-}
\newcommand{\cd}{\delta^{\left\langle 1 \right\rangle}}
\newcommand{\cdd}{\delta^{\left\langle 2 \right\rangle}}

\newcommand{\fa}{\mu^+}

\newcommand{\ca}{\mu^{\left\langle 1 \right\rangle}}

\newcommand{\fdi}{\delta_{\mathrm{FD}}^{-1}}
\newcommand{\Dx}{\Delta x}
\newcommand{\Dt}{\Delta t}

\newcommand{\dx}{\mathrm{d} x}

\newcommand{\NN}{\mathbb{N}}

\begin{document}

\title[Analysis of nonquadratic conservative schemes]{Analysis of nonquadratic energy-conservative schemes for KdV type-equations}

\author{Shuto Kawai}
\address{Department of Mathematical Informatics, Graduate School of Information Science and Technology, University of Tokyo, 7-3-1, Bunkyo-ku, Tokyo 113-8656, Japan}
\curraddr{}
\email{kawai-sk@g.ecc.u-tokyo.ac.jp}
\thanks{}

\author{Shun Sato}
\address{}
\curraddr{}
\email{}
\thanks{}

\author{Takayasu Matsuo}
\address{}
\curraddr{}
\email{}
\thanks{}


\date{}

\dedicatory{}

\begin{abstract}
  Numerical schemes that conserve invariants have demonstrated superior performance in various contexts, and several unified methods have been developed for constructing such schemes. However, the mathematical properties of these schemes remain poorly understood, except in norm-preserving cases.
  This study introduces a novel analytical framework applicable to general energy-preserving schemes. The proposed framework is applied to Korteweg-de Vries (KdV)-type equations, establishing global existence and convergence estimates for the numerical solutions.
\end{abstract}

\maketitle

\section{Introduction}
\label{sec:intro}

This study aims to develop a new mathematical framework for analyzing certain conservative numerical schemes applied to Korteweg--de Vries (KdV)-type equations, represented as:
\begin{align}\label{eq_kdvtype}
  u_t = - \alpha \left( f(u) \right)_x + \beta u_{xxx} \quad \left( t > 0,\, x\in \mathbb{R} \right),
\end{align}
where $\alpha \in \RR$ and $\beta \in \RR\setminus\{0\}$ are  parameters
and $f: \RR \to \RR$ is a continuously differentiable function.
We imposed the following initial condition:
\begin{align*}
  u(0,x) = u_0(x) \quad \left(x \in \mathbb{R}\right),
\end{align*}
and assumed the periodic boundary condition
\begin{align*}
  u(t, x+L) = u(t,x) \quad \left(t>0, x \in \mathbb{R}\right)
\end{align*}
for a period $L > 0$.

For examle, Equation~\eqref{eq_kdvtype} encompasses:
\begin{itemize}
   \item the KdV equation with $f(u) = u^2/2$ (e.g., $\alpha = 6, \beta = 1$),
   \item the generalized KdV equation with $f(u) = u^p/p$ ($p \in \NN$),
   \item the Ostrovsky equation with $f(u) = u^2/2 + \gamma {\partial_x}^{-2} u$ ($\gamma \in \RR$, where ${\partial_x}^{-1}$ denotes a suitably defined generalized inverse of $\partial/\partial x$; see Subsection~\ref{subsec:appl}).
\end{itemize}

Equation~\eqref{eq_kdvtype} exhibits the following invariants:
\begin{align}
  \mathcal{M} (t) &\coloneqq \int_0^L u(t,x) \dx = \mathcal{M}(0),\\
  \mathcal{E} (t) &\coloneqq \int_0^L \left[ \alpha F(u) + \dfrac{\beta}{2} (u_x(t,x))^2 \right] \dx = \mathcal{E} (0),
\end{align}
where $F$ satisfies $\delta F/\delta u = f$. These invariants are referred to as ``mass''($\mathcal{M} (t)$) and ``energy''($\mathcal{E}(t)$).
Additionally, Equation~\eqref{eq_kdvtype} possesses an extra invariant for specific $f(u)$, such as in the KdV and Ostrovsky cases:
\begin{align}
  \mathcal{N} (t) &\coloneqq \int_0^L (u(t,x))^2 \dx = \mathcal{N} (0),
\end{align}
referred to here as ``norm.''
Preserving these invariants in numerical schemes enhances stability and ensures qualitatively accurate solution behavior.
While mass conservation is straightforward in most consistent numerical schemes, energy conservation is significantly more challenging.
Nonetheless, several general methods for achieving energy preservation have been developed~\cite{Celledoni,DVDM}.
These methods produce schemes that preserve both mass and energy.
These methods align with traditional techniques known for decades when applied to norm conservation, as described in~\cite{Turner}.
For instance, the equation in the standard KdV case can be reformulated as:
\begin{align*}
 u_t = -6uu_x -u_{xxx} = \left [-2(u\partial_x + \partial_x u) - \partial_{xxx}\right]u.
\end{align*}
The skew-symmetry of the operator in $[\cdot]$ in $L^2$ ensures norm preservation, and any skew-symmetric discretization of this operator leads to a norm-preserving spatial discretization.
Norm-preserving schemes can be constructed using a Gauss-type Runge--Kutta methods (cf.~\cite{Hairer}).

Their mathematical analysis remains underexplored despite significant progress in developing conservative schemes.
We review studies focused on general, nonconservative numerical schemes, specifically full discrete schemes, and limit the discussion to the standard KdV case.
Dougalis--Karakashian~\cite{Dougalis} investigated nonconservative Galerkin schemes for the KdV equation, establishing the local existence of approximate solutions and deriving a convergence estimate.
In this context, ``local existence'' implies that solutions exist only ``for sufficiently small $\Dt$ depending on approximate solutions.''
This restriction arises because the standard fixed-point argument requires $\Dt$ to remain small.
Consequently, $\Dt$ may approach zero rapidly, preventing the extension of approximate solutions beyond a finite time $T$, even when the solution to the original partial differential equation (PDE) exists globally.
Holden--Koley--Risebro~\cite{Holden1} analyzed an IMEX-type finite difference scheme and established its convergence rate.
The scheme is linearly-implicit scheme; hence, it guarantees the existence of approximate solutions for some $\Dt$.
Furthermore, the scheme facilitates choosing $\Dt>0$ arbitrarily within the range $\Dt < \exists \overline{\Dt}$, enabling the solution to extend to arbitrarily large time $T$. This capability constitutes a ``global existence'' type theorem.
Court\`es--Lagouti\`ere--Rousset~\cite{Courtes} examined another IMEX-type scheme and provided a convergence estimate.

Global existence proofs for fully implicit schemes typically rely on $L^2$-bounds for the approximate solutions, which enable the fixed-point argument to hold for any $\Dt < \exists \overline{\Dt}$.
A classical contribution by Baker--Dougalis--Karakashian~\cite{Baker} introduced a norm-preserving Galerkin scheme for the KdV equation, demonstrating global existence and establishing a convergence estimate.
Wang--Sun~\cite{Wang4} extended these results to a finite difference variant of the Galerkin scheme, obtaining similar outcomes. Additional examples include the work of Shen--Wang--Sun~\cite{Shen2} and Kawai--Sato--Matsuo~\cite{Ostrovsky_norm}, who generalized the results to the Ostrovsky equation.
This analytical framework also applies to other PDEs. For example, studies on ``energy''-preserving schemes for the Rosenau--KdV and the generalized Rosenau--KdV equations in~\cite{Rosenau-KdV,gen_Rosenau-KdV} examined the following quadratic ``energy'' invariant:
\begin{align}
  \int_0^L \left[ (u(t,x))^2 + (u_{xx}(t,x))^2 \right] \dx = \text{Const.}
\end{align}
Similarly, the following quadratic ``energy'' invariant for the Rosenau--KdV--RLW equation (resp. the generalized Rosenau--KdV--RLW equation) was analyzed in~\cite{Rosenau-KdV-RLW} (resp. \cite{gen_Rosenau-KdV-RLW}):
\begin{align}
  \int_0^L \left[ u(t,x)^2 + \delta u_x(t,x)^2 + \lambda u_{xx}(t,x)^2 \right] \dx = \text{Const.} \qquad \mbox{($\delta, \lambda>0$).}
\end{align}
These quadratic invariants consist solely of positive terms, enabling the derivation of $L^2$-bounds.
Such bounds are instrumental in controlling errors arising from nonlinear terms.

The challenges increase significantly when addressing conservative schemes that preserve more general energy invariant.
Preserving $\mathcal{E}$ does not inherently provide $L^2$-bounds because the terms in $\mathcal{E}$ may have different signs, and $F(u)$ is not necessarily quadratic (referred to here as ``non-quadratic energies'' for brevity).
This limitation complicates the mathematical analysis, particularly when establishing stronger theorems.
This situation is paradoxical because energy invariants are the most critical invariants characterizing the PDEs and often provide essential a priori estimates for solutions.
These properties are precisely how preserving energy invariants in numerical solutions is desirable.
In numerical computations, energy-preserving schemes frequently outperform norm-preserving schemes, even when both schemes are feasible.
However, selecting energy invariants to preserve introduces additional difficulties for mathematical analysis.

This study aims to address this long-standing gap, demonstrating that the $L^2$-bound strategy can be extended to general energy invariants.
The methodology involved two primary components.
First, we adopted an induction-based approach, comprising the  following three steps:
(i) Establishing local existence by assuming a local bound on $\| \ud{m}{} \|_{\infty}$ (the sup-norm of numerical solutions, with the notation defined later).
(ii) Deriving a local convergence estimate based on the local existence result.
(iii) Obtaining a stronger a priori estimate on $\| \ud{m+1}{} \|_{\infty}$ using the local convergence estimate.
Although the local bound on $\| \ud{m}{} \|_{\infty}$ established in step (i) is expansive at this point, the final step reduces it to a non-expansive bound,
enabling the induction process to elevate the local results to a global level (a similar inductive argument is discussed in~\cite{Courtes}).
A critical distinction between this framework and norm-preserving case is using $L^{\infty}$-bounds instead of $L^2$-bounds, which substantially simplifies the estimation of nonlinear terms.
$L^\infty$-bounds have been employed in some existing analyses of structure-preserving schemes, including the norm- and energy-preserving scheme for the nonlinear Schr\"odinger equation~\cite{Akrivis}, the dissipative scheme for the Cahn--Hilliard equation~\cite{Furihata}, and conservative schemes for the modified Cammasa--Holm and modified Hunter--Saxton equations \cite{mCH,mHS}.
In such cases, preserving (or bounding) of a specific energy function directly provides $L^\infty$-bounds.
However, the energy functions for KdV-type equations do not directly yield $L^\infty$-bounds.
For instance, both the norm and energy are required in the standard KdV equation to establish an $L^\infty$-bound.
The three-step argument proposed in this study overcomes this limitation by deriving the missing bound through induction.

The second principal contribution of this study is the introduction of a compensating modified energy function argument, which enables fully utilization of the conserved discrete energy as intended.
The energy function incorporated the term on $\| u_x\|_2^2$, expected to facilitate bounding $\|\fd_x \ed{m}{} \|$ (the first-order difference of the errors; notation explained later), a critical step in completing the overall estimation process.
However, the term $F(u)$, which can assume different signs, complicates the argument significantly. This difficulty represents one of the primary reasons why mathematical analyses of energy-preserving schemes remain unresolved.
This study demonstrated that adding a compensatory term to the energy function mitigated the challenging effects of $F(u)$.
The modified energy successfully established the necessary bound on $\|\fd_x \ed{m}{} \|$.

We anticipate that the proposed strategy will benefit other PDEs and their corresponding energy-preserving schemes.

The remainder of this paper is structured as follows;
Section~\ref{sec:pre} introduces the notations and lemmas used throughout the subsequent sections.
Section~\ref{sec:analysis} provides a mathematical analysis using the KdV equation as an illustrative example.


\section{Notations and discrete symbols}
\label{sec:pre}

Numerical solutions are denoted as $ \ud{m}{k}$ $( m=0,\dots, M$; $k \in \mathbb{Z}$), where $ \Dt$ ($\coloneqq T/M$ for some $T\in\mathbb{R}_+$) and $ \Delta x \, (\coloneqq L/K)$ are the temporal and spatial mesh sizes, respectively.
A discrete periodic boundary condition was assumed, such that $ \ud{m}{k+K} = \ud{m}{k} \left ( k \in \mathbb{Z} \right) $.
The notation $ \ud{m}{} := \left( \ud{m}{1} , \dots ,\ud{m}{K} \right)^{\top} $ was also used.

See also~\cite{DVDM} for symbols and lemmas below, .
The spatial forward and backward difference operators are defined as:
\begin{align*}
  \fd_x \ud{m}{k} \coloneqq \frac{ \ud{m}{k+1} - \ud{m}{k} }{\Dx}, \quad \bd_x \ud{m}{k} \coloneqq \frac{\ud{m}{k} - \ud{m}{k-1}}{\Dx}.
\end{align*}
The spatial first and second order central difference operators are defined as:
\begin{align*}
  \cd_x \ud{m}{k} &\coloneqq \frac{ \ud{m}{k+1} - \ud{m}{k-1} }{2\Dx}, \quad \cdd_x \ud{m}{k} \coloneqq \frac{ \ud{m}{k+1} - 2\ud{m}{k} + \ud{m}{k-1} }{(\Dx)^2}.
\end{align*}
The spatial and temporal forward average operators are defined as:
\begin{align*}
  \fa_x \ud{m}{k} \coloneqq \frac{ \ud{m}{k+1} + \ud{m}{k} }{2}, \quad \fa_t \ud{m}{k} \coloneqq \frac{ \ud{m+1}{k} + \ud{m}{k} }{2}.
\end{align*}
The temporal forward difference operator is given by:
\begin{align*}
  \fd_t \ud{m}{k} \coloneqq \frac{\ud{m+1}{k} - \ud{m}{k}}{\Dt}.
\end{align*}

We defined the discrete Lebesgue space $L_K^p \, \left(1 \leq p \leq \infty \right)$ as the pair $\left(\RR^K, \left\| \cdot \right\|_p \right)$, where the norm is:
\begin{align*}
  \|v\|_p \coloneqq \left(\sum_{k=1}^K \left|v_k\right|^p \Dx \right)^{\frac{1}{p}} (1 \leq p < \infty),\quad \left\|v\right\|_\infty \coloneqq \max_{k\in\mathbb{Z}}\left|v_k\right|.
\end{align*}
The associated inner product $\left\langle \cdot, \cdot \right\rangle$ for $L_K^2$ is defined as:
\begin{align*}
  \left\langle v, w\right\rangle \coloneqq \sum_{k=1}^K v_k w_k \Dx.
\end{align*}
We used $\|\cdot\|_2$ to denote $\|\cdot\|$ in the subsequent discussion to simplify the notation.

The following lemma consolidates several fundamental properties.


\begin{lemma}\label{lem_ope}
  The following properties hold:
  \begin{enumerate}
    \item All previously defined operators commute with one another.
    \item $\cd_x = (\fd_x + \bd_x)/2$, $\cdd_x = \fd_x \bd_x$.
    \item $\fd_x ( v_k w_k ) = ( \fd_x v_k ) w_{k+1} + v_k (\fd_x w_k)$.
    \item (skew-symmetry) $\left\langle \fd_x v, w \right\rangle = -\left\langle v,\bd_x w \right\rangle$, $\left\langle \cd_x v, w \right\rangle = -\left\langle v, \cd_x w \right\rangle$.
    \item $2\left\langle v, w \right\rangle \leq \| v \|^2 + \| w \|^2$.
    \item $\left\| v + w \right\|^2 \leq 2(\| v \|^2 + \| w \|^2)$.
  \end{enumerate}
\end{lemma}


\begin{lemma}[Discrete Sobolev Lemma; Lemma 3.2 in \cite{DVDM}]\label{lem_Sobolev}
  For any $v$,
  \begin{align*}
    \left\| v \right\|_\infty \leq \hat{L} \left\| v \right\|_{H^1} \coloneqq \hat{L} \left( \left\| v \right\|^2 + \left\| \fd_x v \right\|^2 \right)^\frac{1}{2},
  \end{align*}
  where $\hat{L} = \sqrt{2} \max \left\{ \sqrt{L}, 1/\sqrt{L} \right\}$.
\end{lemma}

The subsequent lemma, of the Gronwall type, is crucial for deriving error estimates in later sections.

\begin{lemma}\label{lem_gronwall}
  Let $\vd{m}{}$ denote a sequence satisfying $\vd{m}{} \geq 0$ for all $m$ and $\vd{0}{} = 0$.
  If constants $c, d > 0$ exist such that
  \begin{align}\label{eq_gen}
    \fd_t \vd{m}{} \leq c \fa_t \vd{m}{} + d
  \end{align}
  for all $m = 0,\ldots,M_0$, and sufficiently small $\Dt$, then,
  \begin{align*}
    \vd{m}{} \leq 2dT\mathrm{exp}(2cT),
  \end{align*}
  for all $m = 0,\ldots,M_0+1$.
\end{lemma}

\begin{proof}
  This result follows from \cite{gronwall}.
\end{proof}


\section{Mathematical analysis}
\label{sec:analysis}

The scope of this analysis is restricted to the standard KdV equation as follows to ensure clarity:
\begin{align}\label{eq_Zak}
  u_t = - \alpha u u_x + \beta u_{xxx} \quad \left( t > 0,\, x\in \mathbb{R} \right),
\end{align}
where the logical structure of the argument is outlined.
Generalizations are considered later.

This study introduces a new analytical approach, using the KdV equation as an illustrative example.
However, we briefly summarize prior work for completeness.
KdV was developed as a model for shallow-water waves, and the KdV equation has since been applied in diverse areas, including hydrodynamics and plasma physics~\cite{Jeffrey,Lamb,Crighton,Cascaval}.
Results obtained under various settings are well-documented~\cite{Zak-Fad,Miura,Linares,Akhunov,Molinet}.
Many numerical methods have been proposed, encompassing finite difference \cite{Taha,Schiesser,Feng,Ismail,Mazhukin,Holden1,Zhang,Courtes}, finite volume \cite{Dutykh1,Benkhaldoun}, finite element \cite{Arnold,Bona}, Galerkin \cite{Baker,Cai,Dong,Ma}, spectral \cite{Maday,Helal,Korkmaz}, multisymplectic  \cite{Zhao,Wang1,Ascher1,Ascher2,Bridges,Dutykh2}, collocation \cite{Al-Khaled,Sadri} and other methods~\cite{Wang2,Holden2,Kong,Wang3,Shen1,Brugnano,Uzunca}.
Mathematical analyses of structure-preserving schemes have been undertaken in \cite{Baker,Shen2,Wang4}, but these efforts focus exclusively on norm-preserving schemes.
No analysis of energy-preserving schemes has been conducted.

We considered the following energy-conservative scheme for~\eqref{eq_Zak}~\cite{Furihata_JCP,DVDM}:
\begin{align}\label{sch}
  \fd_t \ud{m}{k} = - \dfrac{\alpha}{6} \cd_x \left\{ \left( \ud{m+1}{k} \right)^2 + \ud{m+1}{k} \ud{m}{k} + \left( \ud{m}{k} \right)^2 \right\} + \beta \cd_x \cdd_x \fa_t \ud{m}{k}.
\end{align}
This scheme conserves the mass and energy as follows:
\begin{align*}
  \mathcal{M}_\mathrm{d}(\ud{m}{}) := \sum_{k=1}^K \ud{m}{k} \Dx = \mathcal{M}_\mathrm{d}(\ud{0}{}),
\end{align*}
and
\begin{align}\label{eq_energy}
  \mathcal{E}_\mathrm{d}(\ud{m}{}) &\coloneqq \frac{\alpha}{6} \sum_{k=1}^K \left( \ud{m}{k} \right)^3 \Dx + \frac{\beta}{2} \left\| \fd_x \ud{m}{} \right\|^2 = \mathcal{E}_\mathrm{d}(\ud{0}{}).
\end{align}

\subsection{Overall strategy}
\label{sec:out}

Obtaining $L^2$- or $L^\infty$-bounds was not feasible.
Instead, a three-step approach was employed.

\paragraph{\textup{(i)} Local existence and boundedness of solutions}
Assuming the existence of $\ud{m}{}$ for $m \le M-1$, the proof establishes the unique existence of $\ud{m+1}{}$.
The argument, a modification of the standard method, explicitly evaluates the bound
\begin{align*}
  \left\| \ud{m+1}{} \right\|_\infty \leq q \left\| \ud{m}{} \right\|_\infty \quad \text{and thus} \quad \left\| \ud{m+1}{} \right\|_\infty \leq q r \qquad ( q > 1 ),
\end{align*}
where $r \geq \underset{m'\leq m}{\max} \left\| \ud{m'}{} \right\|_\infty$.
We obtained the following by reiterating the inequality:
\begin{align*}
  \left\| \ud{m}{} \right\|_\infty \leq q^m \left\| \ud{0}{} \right\|_\infty.
\end{align*}
While this diverges as $m \to \infty$, the numerical solutions ($\ud{0}{}, \ud{1}{}, \ldots, \ud{m+1}{}$) are locally bounded.

\paragraph{\textup{(ii)} Local boundedness implies local convergence}
This boundedness aided in deriving a conditional convergence estimate using modified energy conservation to evaluate the difference operators in the scheme.

\paragraph{\textup{(iii)} A stronger a priori estimate}
A stronger estimate was obtained from the local convergence result:
\begin{align*}
  \left\| \ud{m+1}{} \right\|_\infty \leq r,
\end{align*}
requiring only $r \geq \left\| \ud{0}{} \right\|_\infty$.

The subsequent subsections elaborate on these steps, assuming that the underlying KdV solution is global and sufficiently smooth.

\subsection{Local existence}
\label{sec:loc_sol}


\begin{theorem}[local existence]\label{thm_local_solvability}
  Suppose $\ud{0}{}, \ud{1}{}, \ldots, \ud{m}{}$ are solutions of the scheme \eqref{sch}, $q > 1$, and $r$  satisfies
  \begin{align*}
    r \geq \max_{m'\leq m} \left\| \ud{m'}{} \right\|_\infty.
  \end{align*}
  If $\Dt$ and $\Dx$ satisfy $\Dt < \min\left\{\varepsilon_1(q, r,\Dx), \varepsilon_2(q, r,\Dx)\right\}$, with
  \begin{align*}
    \varepsilon_1( q, r,\Dx) &\coloneqq ( q - 1 ) (\Dx)^3 \left[ \dfrac{|\alpha|}{6} (\Dx)^2 ( q^2 + q + 1) r + \frac{3}{2} |\beta| ( q + 1 ) \right]^{-1},\\
    \varepsilon_2( q, r,\Dx) &\coloneqq (\Dx)^3 \left[ \dfrac{|\alpha|}{6} (\Dx)^2 (2 q + 1) r + \frac{3}{2} |\beta| \right]^{-1},
  \end{align*}
  then, the scheme \eqref{sch} admits a unique solution $\ud{m+1}{}$ satisfying
  \begin{align*}
    \left\| \ud{m+1}{} \right\|_\infty \leq q r.
  \end{align*}
\end{theorem}

\begin{proof}
This proof uses a standard fixed-point argument, detailed in Appendix~\ref{app_local_solvability}.
\end{proof}

Theorem~\ref{thm_local_solvability} differs from standard existence theorems by employing the sup norm instead of the $L^2$-norm, while also providing an explicit evaluation of $\left\| \ud{m+1}{} \right\|_\infty$.

\subsection{Conditional Convergence}
\label{sec:con_con}

The following theorem establishes conditional convergence, building on Theorem~\ref{thm_local_solvability}: Let the target solution of the KdV equation be defined as $\tud{m}{} \coloneqq (u(m\Dt,k\Dx))_k$, and the error by expressed as $\ed{m}{} \coloneqq \ud{m}{} - \tud{m}{}$.


\begin{theorem}[conditional convergence]\label{thm_conditional_convergence}
  Under the assumptions of Theorem~\ref{thm_local_solvability} for $m = M_0$ $(0 \leq M_0 < M)$, let $u$ denote a sufficiently smooth solution of the KdV equation, and $r > \sup_{t,x} |u(t,x)|$.
  Hence, for $m = 0, 1, \ldots, M_0+1$, and sufficiently small $\Dt \leq \Dx$, the following inequality holds:
  \begin{align*}
    \left\| \ed{m}{} \right\|_{H^1} \coloneqq \left( \left\| \ed{m}{} \right\|^2 + \left\| \fd_x \ed{m}{} \right\|^2 \right)^{1/2} \leq C \left( (\Dt)^2+(\Dx)^2 \right),
  \end{align*}
  where $C = C( q, r) > 0$ is a constant independent of $\Dt$ and $\Dx$.
\end{theorem}


The assumptions of Theorem~\ref{thm_conditional_convergence} are maintained in this subsection
(and are not repeated in the subsequent lemmas).
The local truncation errors $\td{m}{k}$ are defined by:
\begin{align*}
  \fd_t \tud{m}{k} = - \dfrac{\alpha}{6} \cd_x \left\{ \left( \tud{m+1}{k} \right)^2 + \tud{m+1}{k} \tud{m}{k} + \left( \tud{m}{k} \right)^2 \right\} + \beta \cd_x \cdd_x \fa_t \tud{m}{k} + \td{m}{k}.
\end{align*}

The following lemma provides a critical local truncation error estimate.


\begin{lemma}\label{lem_truncation}
  If the solution of \eqref{eq_Zak} is sufficiently smooth, then the following inequalities hold for $m=0,\ldots,M_0$,
  \begin{align*}
    \left\| \td{m}{} \right\|, \left\| \fd_x \td{m}{} \right\| \leq c_0 \left( (\Dt)^2 + (\Dx)^2 \right),
  \end{align*}
  where $c_0 > 0$ is a constant independent of $\Dt$ and $\Dx$.
\end{lemma}


\begin{proof}
  This lemma follows from the standard argument using the Taylor expansions.
\end{proof}

 The smoothness assumption in Theorem~\ref{thm_conditional_convergence} arises solely from this local estimate. The KdV equation admits sufficiently smooth global solutions~\cite{Molinet}.
 Extending this argument to cases of lower regularity poses an interesting question, though it lies beyond the scope of this study.

The next lemma is essential for the error estimation.


\begin{lemma}\label{lem:e:1}
  For $m = 0,\ldots,M_0$, the following inequality holds:
  \begin{align*}
    \fd_t \left\| \ed{m}{} \right\|^2 \leq C_1 \fa_t \left\| \ed{m}{} \right\|_{H^1}^2 + (c_0)^2 \left( (\Dt)^2+(\Dx)^2 \right)^2,
  \end{align*}
  where $c_0$ is defined in Lemma~\ref{lem_truncation}, and
  \begin{align*}
  C_1 \coloneqq C_1(q,r) \coloneqq \max\left\{  3 |\alpha| ( 2 q^2 + 1) r^2 + 1, |\alpha|/2\right\},
  \end{align*}
  is constant independent of $\Dt$ and $\Dx$.
\end{lemma}


\begin{proof}
Expanding
$\fd_t \left\|\ed{m}{} \right\|^2 = 2 \left\langle \fd_t \ed{m}{}, \fa_t \ed{m}{} \right\rangle$,
\begin{align*}
  \fd_t \ed{m}{k} &= - \dfrac{\alpha}{6} \cd_x \Biggl[ \left\{ \left( \ud{m+1}{k} \right)^2 + \ud{m+1}{k} \ud{m}{k} + \left( \ud{m}{k} \right)^2 \right\}\\
  &\hspace{13mm} - \left\{ \left( \tud{m+1}{k} \right)^2 + \tud{m+1}{k} \tud{m}{k} + \left( \tud{m}{k} \right)^2 \right\} \Biggr] + \beta \cd_x \cdd_x \fa_t \ed{m}{k} - \td{m}{k}\\
  &= - \dfrac{\alpha}{6} \cd_x \left\{ \left( \ed{m+1}{k} \right)^2 + \ed{m+1}{k} \ed{m}{k} + \left( \ed{m}{k} \right)^2 \right\}\\
  &\hspace{5mm} - \dfrac{\alpha}{6} \cd_x \left( 2\ed{m+1}{k} \tud{m+1}{k} + \ed{m+1}{k} \tud{m}{k} + \ed{m}{k} \tud{m+1}{k} + 2 \ed{m}{k} \tud{m}{k} \right)\\
  &\hspace{10mm} + \beta \cd_x \cdd_x \fa_t \ed{m}{k} - \td{m}{k},
\end{align*}
and using straightforward bounding techniques (as outlined in Lemma~\ref{lem_ope}), the claim follows (see Appendix~\ref{app:lem:e:1} for further details).
\end{proof}

This necessitates controlling $\| \fd_x \ed{m}{} \|$. This term can be eliminated in norm-preserving cases~\cite{Wang4, Ostrovsky_norm} by leveraging the specific structure of the schemes.
A similar inequality bound was initially considered in a recent analysis in~\cite{Courtes}.
However, all terms involving $\fd_x \ed{m}{}$ were shown to be ignored due to the non-positivity of their coefficients.
The applicability of this technique to other PDEs and schemes remains uncertain.

Hence, this study introduces a novel approach fully utilizing the discrete conserved energy defined in Equation \eqref{eq_energy}:
\begin{align*}
  \mathcal{E}_\mathrm{d}(\ud{m}{}) &= \frac{\alpha}{6} \sum_{k=1}^K \left( \ud{m}{k} \right)^3 \Dx + \frac{\beta}{2} \left\| \fd_x \ud{m}{} \right\|^2.
\end{align*}
Naturally, the following can be evaluated from this definition:
\begin{align}\label{eq_errorenergy}
  \left\| \fd_x \ed{m}{} \right\|^2 +  A^{(m)},
  \quad
  \mbox{where}\
   A^{(m)} := \frac{\alpha}{3\beta} \sum_{k=1}^K \left( \ed{m}{k} \right)^3 \Dx.
\end{align}
The next lemma establishes the associated inequality.

\begin{lemma}\label{lem:e:2}
  If $\Dt \leq \Dx$, then, for $m = 0,\ldots,M_0$, the following holds,
  \begin{align*}
    &\fd_t \left\| \fd_x \ed{m}{} \right\|^2 \leq - \fd_t A^{(m)} + C_2 \fa_t \left\| \ed{m}{} \right\|_{H^1}^2 + \left( \dfrac{|\alpha|}{6|\beta|} + 1 \right) (c_0)^2 \left( (\Dt)^2+(\Dx)^2 \right)^2,
  \end{align*}
  where $c_0$ is defined in Lemma~\ref{lem_truncation}, and $C_2 = C_2( q, r) > 0$ is a constant independent of $\Dt$ and $\Dx$.
\end{lemma}

The proof and explicit form of $C_2$ are provided in Subsection~\ref{subsec:proof_lemma}.

Combining Lemma~\ref{lem:e:1} and Lemma~\ref{lem:e:2} yields:
\begin{align*}
  \fd_t \left( \left\| \ed{m}{} \right\|_{H^1}^2 + A^{(m)} \right) \leq C' \fa_t\left\| \ed{m}{} \right\|_{H^1}^2 + \left( \dfrac{|\alpha|}{6|\beta|} + 2 \right) (c_0)^2 \left( (\Dt)^2+(\Dx)^2 \right)^2,
\end{align*}
where $C' \coloneqq C_1+C_2$.
Lemma~\ref{lem_gronwall} cannot be applied because $A^{(m)}$ may take negative values.
This poses a significant difficulty in analyzing energy functions that are neither quadratic nor positive.
Hence the analysis employed a ``modified'' energy function, which compensated the negative term:
\begin{align*}
 \mathcal{E}_\mathrm{d}' &\coloneqq \theta \left\| \ed{m}{} \right\|^2 + \mathcal{E}_\mathrm{d} = \theta \left\| \ed{m}{} \right\|^2+ \left\| \fd_x \ed{m}{} \right\|^2 + A^{(m)},
\end{align*}
where $\theta > 0$ was chosen sufficiently large in the subsequent lemma to neutralize $A^{(m)}$.
The next lemma specifies the conditions under which $ \mathcal{E}_\mathrm{d}'$ is sufficient for the analysis instead of $\left\| \ed{m}{} \right\|_{H^1}^2$.


\begin{lemma}\label{lem:e:bound}
  For $m = 0,\ldots,M_0+1$, $0 \leq \left\| \ed{m}{} \right\|_{H^1}^2 \leq \mathcal{E}_\mathrm{d}'$ holds if $\theta \geq 1 +  2pr|\alpha|/3|\beta|$.
\end{lemma}


\begin{proof}
  Note that $\left\| \ed{m}{} \right\|_\infty \leq \left\| \ud{m}{} \right\|_\infty + \left\| \tud{m}{} \right\|_\infty \leq 2 q r$ holds by the assumption.
  Hence, the following inequality holds:
  \begin{align*}
    \left\vert A^{(m)} \right\vert \leq \dfrac{|\alpha|}{3|\beta|} \left\| \ed{m}{} \right\|_\infty \left\| \ed{m}{} \right\|^2 \leq \dfrac{|\alpha|}{3|\beta|} \cdot 2 q r \left\| \ed{m}{} \right\|^2 ,
  \end{align*}
  which implies:
  \begin{align*}
    \mathcal{E}_\mathrm{d}' \geq \theta \left\| \ed{m}{} \right\|^2 + \left\| \fd_x \ed{m}{} \right\|^2 - \dfrac{2|\alpha|}{3|\beta|} q r \left\| \ed{m}{} \right\|^2 \geq \left\| \ed{m}{} \right\|_{H^1}^2 \geq 0.
  \end{align*}
\end{proof}


The proof of Theorem~\ref{thm_conditional_convergence} can now be completed. We have
\begin{align*}
  \fd_t \mathcal{E}_\mathrm{d}' &= \theta \cdot \fd_t \left\| \ed{m}{} \right\|^2 + \fd_t \left( \left\| \fd_x \ed{m}{} \right\|^2 + A^{(m)} \right)\\
  &\leq \left( \theta C_1 + C_2 \right) \fa_t \left\| \ed{m}{} \right\|_{H^1}^2 + \left( \dfrac{|\alpha|}{6|\beta|} + \theta \right) (c_0)^2 \left( (\Dt)^2+(\Dx)^2 \right)^2\\
  &\leq \left( \theta C_1 + C_2 \right) \fa_t \mathcal{E}_\mathrm{d}' + \left( \dfrac{|\alpha|}{6|\beta|} + \theta \right) (c_0)^2 \left( (\Dt)^2+(\Dx)^2 \right)^2
\end{align*}
for $m = 0,\ldots,M_0$.
Therefore, Lemma~\ref{lem_gronwall} for $m = 0,\ldots,M_0+1$ implies:
\begin{align*}
  \left\| \ed{m}{} \right\|_{H_1}^2 \leq \mathcal{E}_\mathrm{d}' &\leq
  \left[
  2 \left( \dfrac{|\alpha|}{6|\beta|} + \theta \right) (c_0)^2T
  \mathrm{exp} \left( 2 \left( \theta C_1 + C_2 \right) T \right)
  \right]
  \left( (\Dt)^2 + (\Dx)^2 \right)^2
\end{align*}
holds for sufficiently small $\Dt \leq \Dx$,
where $C > 0$ is the constant within $[\,\cdot\,]$.

\subsection{An a priori estimate on $\|\ud{m+1}{}\|$}

We can derive a priori estimate for $\left\| \ud{m+1}{} \right\|_\infty$ from the conditional convergence, enabling the use of an induction argument.

\begin{corollary}[a priori estimate]\label{cor_apriori}
  Under the assumption of Theorem~\ref{thm_conditional_convergence}, if
  \begin{align} \label{res_dx}
    \Dx \leq \left( \dfrac{ r - \sup_{(t,x)\in[0,T]\times [0,L]}|u(t,x)|}{2\hat{L}C( q, r)} \right)^{1/2},
  \end{align}
  where $\hat{L} \coloneqq \sqrt{2} \max\{\sqrt{L}, 1/\sqrt{L} \}$, then, $\left\| \ud{m}{} \right\|_\infty \leq r$ for $m = 0, 1, \ldots, M_0+1$.
\end{corollary}

\begin{proof}
  Given this assumption, we only need to consider the case $m = M_0+1$. Applying the discrete Sobolev Lemma (Lemma~\ref{lem_Sobolev}),
  \begin{align*}
    \left\| \ud{M_0+1}{} \right\|_\infty \leq \left\| \ed{M_0+1}{} \right\|_\infty + \left\| \tud{M_0+1}{} \right\|_\infty &\leq \hat{L} \left\| \ed{M_0+1}{} \right\|_{H^1} + \sup_{t,x}|u(t,x)|\\
    &\leq 2 \hat{L} C( q, r) (\Dx)^2 + \sup_{t,x}|u(t,x)|.
  \end{align*}
  The right-hand side of the inequality is bounded by $r$ using $\sup|u|<r$ from Theorem~\ref{thm_conditional_convergence}, provided that $\Dx$ is sufficiently small, as specified in the statement.
\end{proof}

\subsection{Global results by induction}

Now, we can present the desired global results.

\begin{theorem}[global existence and convergence]
  Let $ q > 1$.
  For any $T>0$,
  suppose $N$ is chosen sufficiently large such that
  $\Dx \leq \min \left\{ \sqrt{ r/4 \hat{L} C( q, r) }, 1 \right\}$,
  where $r \coloneqq 2 \sup_{(t,x)\in[0,T]\times [0,L]} |u(t,x)|$.
  Some $\overline{\Dt}>0$ exists,
  and for any $\Dt < \overline{\Dt}$,
  the scheme has a unique solution $\ud{m}{}$ $(m = 1, 2, \ldots, M)$, satisfying
  \begin{align*}
    \left\| \ud{m}{} \right\|_\infty \leq r, \quad \left\| \ed{m}{} \right\|_{H^1} \leq C( q, r) \left( (\Dt)^2 + (\Dx)^2 \right) \qquad (m = 1, 2, \ldots, M).
  \end{align*}

\end{theorem}

\begin{proof}
  We chose $\Dt$ smaller than
  $\min\{ \varepsilon_1( q, r, \Dx), \varepsilon_2( q, r, \Dx), \Dx\}$ based on the assumptions in Theorem~\ref{thm_local_solvability} and Theorem~\ref{thm_conditional_convergence}, where $\varepsilon_1, \varepsilon_2, C$ are constants derived from the previous theorems.
  The constants $C(q,r)$, detailed at the end of Section~\ref{sec:con_con}, depend on $c_0, C_1, C_2, \theta, T$ as $T\to\infty$.
  Among these, $C_1(q,r)$ (Lemma~\ref{lem:e:1}) and $\theta$ (Lemma~\ref{lem:e:bound}) remain constant,
  while $c_0$ (Lemma~\ref{lem_truncation}) and $C_2(q,r)$ (Subsection~\ref{subsec:proof_lemma}) may vary bounded by the supremums of (the absolute values of) $u(t,x)$ or its derivatives on $[0,T]\times [0,L]$.
  Therefore, if the target solution is global, these quantities remain bounded.
  Consequently, both $r$ and $C$ can grow (both depending on the growing $T$ and $ r=2\sup|u|$) as $T\to\infty$
  However, they are finite constants for any finite $T$.
  Hence, $\varepsilon_1$, $\varepsilon_2$, and $\min \left\{ \sqrt{ r/(4\hat{L}C)}, 1 \right\}$ are positive constants.
  This implies the existence of some $\overline{\Dt}>0$ smaller than these constants, enabling to choose an arbitrary $\Dt  < \overline{\Dt}$.

  The rest of the proof proceeds by induction.
  We have $\left\| \ud{0}{} \right\|_\infty \leq \sup_{x} |u(0,x)| < r$ for the initial condition.
  Hence, $\ud{1}{}$ exists uniquely and satisfy $\left\| \ud{1}{} \right\|_\infty < q r$ by Theorem~\ref{thm_local_solvability}.
  Therefore, convergence holds for $m=0,1$ by Theorem~\ref{thm_conditional_convergence} and Corollary~\ref{cor_apriori}, and we obtained $\left\| \ud{1}{} \right\|_\infty < r$.
  This process can be repeated for $m=1, 2, \ldots, M$.
\end{proof}

\begin{remark}
  The factor of $2$ in $r \coloneqq 2 \sup_{t,x} |u(t,x)|$ is arbitrary because the results in Theorem~\ref{thm_local_solvability} and Corollary~\ref{cor_apriori} hold regardless of this choice.
  Let us denote this factor by $d$, i.e., $r_0 \coloneqq \sup_{t,x} |u(t,x)|$ and $r = d r_0$.
  The condition for $\Dx$ becomes $ \Dx \leq \min \left\{ \sqrt{ (d-1)r_0/2 \hat{L} C( q, r) }, 1 \right\} $ in Corollary~\ref{cor_apriori}.
  If $d\to 1$, sharper estimates for $\|\ud{m}{}\|_{\infty}$ can be achieved.
  However, this imparts a stricter restriction on $\Dx$ due to the imposed condition.
  Conversely, the limit $d\to\infty$ would relax the restriction on $\Dx$.
  Nevertheless, the final outcome remains effectively unchanged since $C_1, C_2$ in $C(q,r)$ scale as $O(r^2)$.
  This trade-off implies the existence of an optimal value for $d$ that maximizes the upper bound of $\Dx$.
  However, further exploration of this optimal factor is beyond the scope of this discussion.
\end{remark}

\subsection{Generalization to other PDEs} \label{subsec:appl}

The proof approach outlined in Theorem~\ref{thm_conditional_convergence} applies to other PDEs and their conservative schemes.

\subsubsection{Ostrovsky case}

An energy-conservative scheme for the Ostrovsky equation is presented in \cite{Ostrovsky_sch}:
\begin{align}
  \begin{split}
    \fd_t \ud{m}{k} &= - \dfrac{\alpha}{6} \cd_x \left\{ \left( \ud{m+1}{k} \right)^2 + \ud{m+1}{k} \ud{m}{k} + \left( \ud{m}{k} \right)^2 \right\}\\
    &\hspace{10mm} + \beta \cd_x \cdd_x \fa_t \ud{m}{k} + \gamma \cd_x (\fdi)^2 \fa_t \ud{m}{k},
  \end{split}
\end{align}
where $\fdi$ denotes a discrete generalized inverse operator (see~\cite{Ostrovsky_sch} for its definition).
The associated energy is given by:
\begin{align*}
    \mathcal{E}(\ud{m}{}) \coloneqq \dfrac{\alpha}{3\beta} \sum_{k=1}^K \left( \ud{m}{k} \right)^3 \Dx + \left\| \fd_x \ud{m}{} \right\|^2 + \dfrac{\gamma}{\beta} \left\| \fdi \ud{m}{} \right\|^2.
\end{align*}
When $\gamma=0$, the equation reduces to the standard KdV case.
Let us define $A^{(m)} = \dfrac{\alpha}{3\beta} \sum_k \left( \ed{m}{k} \right)^3 \Dx + \dfrac{\gamma}{\beta} \left\| \fdi \ed{m}{}\right\|^2 $ for $\gamma \neq 0$.
Using the estimate $\left\| \fdi \right\| \leq L/4$ (see \cite{Ostrovsky_norm}), we obtained:
\begin{align}
\left\vert A^{(m)} \right\vert
  \leq \dfrac{32|\alpha| q r + 3|\gamma| L^2 }{48|\beta|}  \left\| \ed{m}{} \right\|^2   .
\end{align}
After that, Lemma~\ref{lem:e:bound} is replaced by the following result ($\mathcal{E}_\mathrm{d}'$ is defined accordingly):

\begin{lemma}[Lemma~\ref{lem:e:bound} for the Ostrovsky case]
  If $\Dt \leq \Dx $, then for $m = 0,\ldots,M_0+1$, the inequality $0 \leq \left\| \ed{m}{} \right\|_{H^1}^2 \leq \mathcal{E}_\mathrm{d}'$ holds if $\theta \geq 1 +  \dfrac{32|\alpha| q r + 3|\gamma| L^2 }{48|\beta|}$.
\end{lemma}

Additionally, Lemma~\ref{lem_truncation} (local truncation error) and Lemma~\ref{lem:e:2} (preliminary $L^2$-error estimate) must be adjusted, which can be done straightforwardly.
These modifications yield similar global solvability and convergence results.

\subsubsection{Generalized KdV case}

Using the DVDM approach~\cite{DVDM}, the following scheme was obtained:
\begin{align}
  \fd_t \ud{m}{k} = - \dfrac{\alpha}{p(p+1)} \cd_x \left\{
    \dfrac{ \left( \ud{m+1}{k} \right)^{ p + 1} -  \left( \ud{m}{k} \right)^{p+1}} { \ud{m+1}{k} - \ud{m}{k}}
    \right\}
    + \beta \cd_x \cdd_x \fa_t \ud{m}{k} ,
\end{align}
which preserves
\begin{align*}
    \mathcal{E}(\ud{m}{}) \coloneqq - \dfrac{2\alpha}{p(p+1)\beta} \sum_{k=1}^K \left( \ud{m}{k} \right)^{p+1} \Dx + \left\| \fd_x \ud{m}{} \right\|^2 .
\end{align*}
Setting $A^{(m)} = -\dfrac{2\alpha}{p( p+1)\beta} {\displaystyle \sum_{k=1}^K} \left( \ed{m}{} \right)^{p+1} \Dx $, we obtained:
\begin{align}
   \left\vert A^{(m)} \right\vert
  \leq \dfrac{2^p |\alpha| q^{ p-1} r^{ p-1}}{ p ( p + 1)|\beta|}  \left\| \ed{m}{} \right\|^2 .
\end{align}
The corresponding replacement for Lemma~\ref{lem:e:bound} is as follows:

\begin{lemma}[Lemma~\ref{lem:e:bound} for the generalized KdV case]
  If $\Dt \leq \Dx$, then for $m = 0,\ldots,M_0+1$, $0 \leq \left\| \ed{m}{} \right\|_{H^1}^2 \leq \mathcal{E}_\mathrm{d}'$ holds if $\theta \geq 1 +  \dfrac{2^p|\alpha| q^{ p-1} r^{ p-1}}{ p( p + 1)|\beta|}$.
\end{lemma}

Lemma~\ref{lem_truncation} and Lemma~\ref{lem:e:2} must be adjusted, but these modifications are straightforward due to the $\|\ed{m}{}\|_{\infty}$ bound and the manageable treatment of polynomial nonlinearities.
These adjustments yield the desired global results.

\subsection{Proof of~Lemma\ref{lem:e:2}} \label{subsec:proof_lemma}

\begin{proof}
  Let us consider the temporal difference ($\fd_t$) of the quantity in Equation~\eqref{eq_errorenergy}.
  Its first term is expressed as:
  \begin{align*}
    \fd_t \left\| \fd_x \ed{m}{} \right\|^2 &= 2 \left\langle \fd_t \fd_x \ed{m}{}, \fa_t \fd_x \ed{m}{} \right\rangle\\
    \begin{split}
      &= -\dfrac{\alpha}{3} \left\langle \bar{e}^{(m+1/2)}, \fa_t \cd_x \cdd_x \ed{m}{} \right\rangle\\
      &\hspace{4mm} + \dfrac{\alpha}{3} \left\langle \cd_x f^{(m+1/2)}, \fa_t \cdd_x \ed{m}{} \right\rangle - 2\left\langle \fd_x \td{m}{}, \fa_t \fd_x \ed{m}{} \right\rangle,
    \end{split}
  \end{align*}
  where the following notations are adopted:
  \begin{align*}
    \bar{e} &\coloneqq \bar{e}^{(m+1/2)} \coloneqq \left( \ed{m+1}{} \right)^2 + \ed{m+1}{} \ed{m}{} + \left( \ed{m}{} \right)^2,\\
    f &\coloneqq f^{(m+1/2)} \coloneqq 2\ed{m+1}{} \tud{m+1}{} + \ed{m+1}{} \tud{m}{} + \ed{m}{} \tud{m+1}{} + 2 \ed{m}{} \tud{m}{}.
  \end{align*}
  The Hadamard product $(v \ast w)_k := v_kw_k$ is abbreviated as $v\ast w = vw$ and $v\ast v=v^2$ for notational simplicity.
  The skew-symmetry property simplifies the computation: $\left\langle \cd_x \cdd_x \fa_t \ed{m}{}, \cdd_x \fa_t \ed{m}{} \right\rangle = 0$.
  The second term is evaluated as follows:
  \begin{align*}
    \fd_t \sum_{k=1}^K \left( \ed{m}{k} \right)^3 \Dx &= \left\langle \fd_t \ed{m}{}, \bar{e}^{(m+1/2)} \right\rangle\\
    &= - \dfrac{\alpha}{6} \left\langle \cd_x f, \bar{e} \right\rangle + \beta \left\langle \cd_x \cdd_x \fa_t \ed{m}{}, \bar{e} \right\rangle - \left\langle \td{m}{}, \bar{e} \right\rangle,
  \end{align*}
  which similarly used skew-symmetry to eliminate redundant terms.

  Combining these evaluations, we derived the following expression:
  \begin{align*}
    \fd_t \left( \left\| \fd_x \ed{m}{} \right\|^2 + \dfrac{\alpha}{3\beta} \sum_{k=1}^K \left( \ed{m}{k} \right)^3 \Dx \right)
    &= \underbrace{- \dfrac{\alpha^2}{18\beta} \left\langle \cd_x f, \bar{e} \right\rangle}_{I_1} +
    \underbrace{\dfrac{\alpha}{3} \left\langle \cd_x f, \fa_t \cdd_x \ed{m}{} \right\rangle}_{I_2}\\
    &\hspace{6mm} \underbrace{- \dfrac{\alpha}{3\beta} \left\langle \td{m}{}, \bar{e} \right\rangle - 2\left\langle \fd_x \td{m}{}, \fa_t \fd_x \ed{m}{} \right\rangle}_{I_3}.
  \end{align*}
  We note the following to evaluate the terms $I_1$ and $I_3$:
  \begin{align*}
    \left\| \bar{e} \right\|^2 &= \left\| \ed{m+1}{} \left( \ed{m+1}{} + \frac{1}{2} \ed{m}{} \right) + \ed{m}{} \left( \ed{m}{} + \frac{1}{2} \ed{m+1}{} \right) \right\|^2\\
    &\leq 2 \left\| \ed{m+1}{} + \frac{1}{2} \ed{m}{} \right\|_\infty^2 \left\| \ed{m+1}{} \right\|^2 + 2 \left\| \ed{m}{} + \frac{1}{2} \ed{m+1}{} \right\|_\infty^2 \left\| \ed{m}{} \right\|^2\\
    &\leq 2 (3 q r)^2 \left( \left\| \ed{m+1}{} \right\|^2 + \left\| \ed{m}{} \right\|^2 \right) \leq 36 q^2 r^2 \fa_t \left\| \ed{m}{} \right\|^2.
  \end{align*}
  Hence, $I_1$ and $I_3$ can be easily evaluated as:
  \begin{align*}
    I_1 &= - \dfrac{\alpha^2}{18\beta} \sum_{j=0}^{1} \left\langle \cd_x \left(2 \ed{m+j}{} \tud{m+j}{} + \ed{m+1-j}{} \tud{m+j}{} \right), \bar{e}^{(m+1/2)} \right\rangle\\
    &= - \dfrac{\alpha^2}{9\beta} \sum_{j=0}^{1} \left\langle \cd_x \ed{m+j}{} \ca_x \tud{m+j}{} + \ca_x \ed{m+j}{} \cd_x \tud{m+j}{}, \bar{e}^{(m+1/2)} \right\rangle\\
    &\hspace{4mm} - \dfrac{\alpha^2}{18\beta} \sum_{j=0}^1 \left\langle \cd_x \ed{m+1-j}{} \ca_x \tud{m+j}{} + \ca_x \ed{m+1-j}{} \cd_x \tud{m+j}{}, \bar{e}^{(m+1/2)} \right\rangle\\
    &\leq \dfrac{\alpha^2}{|\beta|} \left[ 12 q^2 r^2 + \frac{1}{6} \left( \sup_{t,x} u_x(t,x) \right)^2 \right] \fa_t \left\| \ed{m}{} \right\|^2 + \dfrac{\alpha^2}{6|\beta|} r^2 \fa_t \left\| \fd_x \ed{m}{} \right\|^2
  \end{align*}
  and
  \begin{align*}
    I_3 & \leq \dfrac{|\alpha|}{6|\beta|} \left[ \left\| \td{m}{} \right\|^2 + \left\| \bar{e}^{(m+1/2)} \right\|^2 \right] + \left\| \fd_x \td{m}{} \right\|^2 + \left\| \fa_t \fd_x \ed{m}{} \right\|^2\\
    & \leq \dfrac{6|\alpha|}{|\beta|} q^2 r^2 \fa_t \left\| \ed{m}{} \right\|^2 + \fa_t \left\| \fd_x \ed{m}{} \right\|^2 + \dfrac{|\alpha|}{6|\beta|} \left\| \td{m}{} \right\|^2 + \left\| \fd_x \td{m}{} \right\|^2.
  \end{align*}

  Therefore, the following remains to be evaluated:
  \begin{align*}
    I_2 &= \dfrac{\alpha}{3} \sum_{j=0}^1 \left\langle \cd_x \left( 2\ed{m+j}{} \tud{m+j}{} + \ed{m+1-j}{} \tud{m+j}{} \right), \fa_t \cdd_x \ed{m}{} \right\rangle\\
    &= \dfrac{2\alpha}{3} \sum_{j=0}^1 \left\langle \ca_x \ed{m+j}{} \cd_x \tud{m+j}{}, \fa_t \cdd_x \ed{m}{} \right\rangle\\
    &\hspace{5mm} + \dfrac{\alpha}{3} \sum_{j=0}^1 \left\langle \ca_x \ed{m+1-j}{} \cd_x \tud{m+j}{}, \fa_t \cdd_x \ed{m}{} \right\rangle\\
    & \hspace{5mm} + \dfrac{\alpha}{3} \sum_{j=0}^1 \left\langle \cd_x \ed{m+j}{} \ca_x \left( 2 \tud{m+j}{} + \tud{m+1-j}{} \right), \fa_t \cdd_x \ed{m}{} \right\rangle\\
    &\eqqcolon J_1 + J_2 + J_3.
  \end{align*}
  Subsequently, $J_1$ and $J_2$ are addressed through summation-by-parts, yielding:
  \begin{align*}
    J_1 &= \sum_{j=0}^1 \left\langle \ca_x \ed{m+j}{} \cd_x \tud{m+j}{}, \fa_t \cdd_x \ed{m}{} \right\rangle\\
    &= - \dfrac{2\alpha}{3} \sum_{j=0}^1 \left\langle \fd_x \left( \ca_x \ed{m+j}{} \cd_x \tud{m+j}{}\right), \fa_t \fd_x \ed{m}{} \right\rangle\\
    &= - \dfrac{2\alpha}{3} \sum_{j=0}^1 \left\langle \fd_x \ca_x \ed{m+j}{} \fa_x \cd_x \tud{m+j}{} + \fa_x \ca_x \ed{m+j}{} \fd_x \cd_x \tud{m+j}{}, \fa_t \fd_x \ed{m}{} \right\rangle\\
    &\leq \frac{2|\alpha|}{3} \sup_{t,x} |u_{xx}(t,x)|^2 \fa_t \left\| \ed{m}{} \right\|^2 + \frac{2|\alpha|}{3} \left[ 2 + \sup_{t,x} |u_x(t,x)|^2 \right] \fa_t \left\| \fd_x \ed{m}{} \right\|^2,
  \end{align*}
  and
  \begin{align*}
    J_2 \leq \frac{|\alpha|}{3} \sup_{t,x} |u_{xx}(t,x)|^2 \fa_t \left\| \ed{m}{} \right\|^2 + \frac{|\alpha|}{3} \left[ 2 + \sup_{t,x} |u_x(t,x)|^2 \right] \fa_t \left\| \fd_x \ed{m}{} \right\|^2.
  \end{align*}

  The final term, $J_3$, presents challenges due to the presence of $\cd_x \ed{m+j}{}$ and $\cdd_x \ed{m}{}$ in the inner product.
  Hence, it does not immediately yield an estimate in terms of $\|\fd_x \ed{m}{}\|^2$.
  The following identity which holds for any $K$-periodic sequences $a, b$ to circumvent this difficulty:
  \begin{align*}
    \left\langle a \cd_x b, \cdd_x b \right\rangle &= - \dfrac{1}{2} \left\langle \fd_x \left( a \fd_x b + a \bd_x b \right), \fd_x b \right\rangle\\
    &= - \dfrac{1}{2} \left\langle a_+ \fd_x \fd_x b + \fd_x a \fd_x b, \fd_x b \right\rangle - \dfrac{1}{2} \left\langle a \fd_x \bd_x b + \fd_x a \bd_x b_+, \fd_x b \right\rangle\\
    &= - \left\langle a \cdd_x b, \cd_x b \right\rangle - \left\langle \fd_x a \fd_x b, \fd_x b \right\rangle.
  \end{align*}
  We omitted the spatial sub-index $k$ to conserve space. The abbreviations $a_+ := a_{k+1}$, and $b_+ :=b_{k+1}$ are introduced.
  Using these notations, we derived the following expression:
  \begin{align}\label{eq_idea}
    \left\langle a \cd_x b, \cdd_x b \right\rangle &= - \dfrac{1}{2} \left\langle \fd_x a \fd_x b, \fd_x b \right\rangle \leq \dfrac{1}{2} \left\| \fd_x a \right\|_\infty \left\| \fd_x b \right\|^2,
  \end{align}
  which is effective in reducing the spatial differences in $b$.
  Applying this inequality, we further obtain:
  \begin{align*}
    J_3&= \dfrac{\alpha}{3} \sum_{j=0}^1 \left\langle \cd_x \ed{m+j}{} \ca_x \left( 2 \tud{m+j}{} + \tud{m+1-j}{} \right), \fa_t \cdd_x \ed{m}{} \right\rangle\\
    &= \dfrac{\alpha}{3} \sum_{j=0}^1 \left\langle \cd_x \ed{m+j}{} \ca_x \left\{ 3 \fa_t \tud{m}{} - \frac{(-1)^j}{2} \left( \tud{m+1}{} - \tud{m}{} \right) \right\}, \fa_t \cdd_x \ed{m}{} \right\rangle\\
    &= 2\alpha \left\langle \cd_x \fa_t \ed{m}{} \ca_x \fa_t \tud{m}{}, \fa_t \cdd_x \ed{m}{} \right\rangle\\
    &\hspace{2mm} + \dfrac{\alpha}{6} \left\langle \Dt \fa_t \cdd_x \ed{m}{} \ca_x \fd_t \tud{m}{} , \cd_x \ed{m+1}{} - \cd_x \ed{m}{} \right\rangle\\
    &\leq |\alpha| \left\| \fd_x \ca_x \fa_t \tud{m}{} \right\|_\infty \left\| \fd_x \fa_t \ed{m}{} \right\|^2\\
    &\hspace{2mm} + \dfrac{|\alpha|}{12} \left( 2 \left\| \ca_x \fd_t \tud{m}{} \right\|_\infty^2 \left\| \Dt \bd_x \right\|^2 \left\| \fa_t \fd_x \ed{m}{} \right\|^2 + \left\| \cd_x \ed{m+1}{} \right\|^2 + \left\| \cd_x \ed{m}{} \right\|^2 \right)\\
    &\leq |\alpha| \left[ \sup_{t,x} |u_x(t,x)| + \dfrac{2}{3} \left( \sup_{t,x} |u_t(t,x)| \right)^2 + \dfrac{1}{6} \right] \fa_t \left\| \fd_x \ed{m}{} \right\|^2,
  \end{align*}
  under the condition $\Dt \leq \Dx$ (it is explicitly used in the final inequality).
  The additional spatial difference operator is managed using two approaches: (1) by redistributing it to $\tilde{u}$ via \eqref{eq_idea}, and (2) by canceling it with the term $\tud{m+1}{k} - \tud{m}{k} = \mathrm{O}(\Dt)$.

  Combining all these evaluations yields estimate:
  \begin{align*}
    \fd_t \left( \left\| \fd_x \ed{m}{} \right\|^2 + A^{(m)} \right) &= I_1 + I_2 + I_3\\
    &\leq C_2 \fa_t \left\| \ed{m}{} \right\|_{H^1}^2 +  \left( \dfrac{|\alpha|}{6|\beta|} + 1 \right) (c_0)^2 \left( (\Dx)^2+(\Dt)^2 \right)^2,
  \end{align*}
  which holds for a constant $C_2>0$ (independent of $\Dt$ and $\Dx$). Here,
  \begin{align*}
    C_2 \coloneqq \max \Biggl\{
    & \dfrac{6|\alpha|}{|\beta|} \left( 1 + 2|\alpha| \right) q^2 r^2 + \dfrac{\alpha^2}{2|\beta|} \sup_{t,x} | u_x(t,x) |^2 + |\alpha| \left[ 2 + \sup_{t,x} |u_{xx}(t,x)|^2 \right],\\
    & 1 + \dfrac{\alpha^2}{2|\beta|} r^2 + |\alpha| \left[ \dfrac{3}{2} \sup_{t,x} |u_x(t,x)|^2 + \dfrac{2}{3} \sup_{t,x} |u_x(t,x)|^2 + \dfrac{2}{3} \right] \Biggr\}
  \end{align*}
  under the condition $\Dt \leq \Dx$.
\end{proof}


\section{Concluding remarks}
\label{sec:conclude}

This study introduces a novel argument strategy for theoretically analyzing general energy-preserving schemes for KdV-type equations.
The strategy is centered on employing an induction approach for local estimates and incorporating compensation for modified energy functions.
A detailed argument is presented for the standard KdV case, with additional discussion demonstrating that similar methodologies can be extended to the Ostrovsky and generalized KdV cases.
We conclude that this new argument strategy applies highly to a wide range of energy-preserving schemes.

\appendix
\section{Detailed proof of Theorem~\ref{thm_local_solvability}}
\label{app_local_solvability}
  Let $u_k \coloneqq \ud{m}{k}$ for simplicity. From the definition of the scheme, the function
  \begin{align*}
    \phi_k(w) &\coloneqq u_k + \Dt\left[ - \dfrac{\alpha}{6} \cd_x \left( w_k ^2 + w_k u_k + u_k^2 \right) + \frac{\beta}{2} \cd_x \cdd_x \left( w_k + u_k \right) \right]
  \end{align*}
  is defined, with $\phi(w) \coloneqq \left( \phi_k(w) \right)_k$.
  Hence, $w^\ast = \ud{m+1}{}$ if and only if $w^\ast$ is a fixed point of $\phi$.
  Therefore, providing that $\phi^{(m)}$ is a contraction mapping on
  \begin{align*}
    B(q,r) \coloneqq \left\{ w \relmiddle{|} \left\| w \right\|_\infty \leq q r \right\}.
  \end{align*}
  suffices.
  Specifically, the following must be verified: (i) $\phi(B(q,r)) \subseteq B(q,r)$ and (ii) $\phi$ is a contraction mapping.
  (i) First, we show that, if $\Dt < \varepsilon_1( q , r,\Dx)$, then $\phi(B( q, r)) \subseteq B( q, r)$ holds.
  For $w\in B( q, r)$,
  \begin{align*}
    \left| \phi_k(w) \right| &\leq \left\| u \right\|_\infty + \Dt \Biggl[ \dfrac{|\alpha|}{6} \frac{1}{\Dx} \left( \left\| w \right\|_\infty^2 + \left\| w \right\|_\infty \left\| u \right\|_\infty + \left\| u \right\|_\infty^2 \right)\\
    &\hspace{40mm} + \frac{|\beta|}{2} \frac{3}{(\Dx)^3} \left( \left\| w \right\|_\infty + \left\| u \right\|_\infty \right) \Biggr]\\
    &\leq r + \frac{\Dt}{(\Dx)^3} \left[ \frac{|\alpha|}{6} (\Dx)^2 ( q^2 + q + 1) r^2 + \frac{3}{2} |\beta| ( q+1) r \right]
  \end{align*}
  holds. Therefore, if
  \begin{align*}
    \Dt \leq ( q-1)(\Dx)^3 \left[ \frac{|\alpha|}{6} (\Dx)^2 ( q^2 + q + 1) r + \frac{3}{2} |\beta| ( q+1) \right]^{-1} = \varepsilon_1( q, r),
  \end{align*}
  then
  \begin{align*}
    &\left\| \phi(w) \right\|_\infty \leq r + \frac{\Dt}{(\Dx)^3} \left[ \frac{|\alpha|}{6} (\Dx)^2 ( q^2 + q + 1) r + \frac{3}{2} |\beta| ( q+1) \right] r \leq q r.
  \end{align*}
  (ii) Additionally, we show that, $\phi$ is a contraction mapping if $\Dt < \varepsilon_2( q, r, \Dx)$.
  For $w, \bar{w} \in B( q, r)$,
  \begin{align*}
    \phi_k(w) - \phi_k(\bar{w})
    &= \Dt\left[ - \dfrac{\alpha}{6} \cd_x \left\{ w_k \left(w_k + u_k \right) - \bar{w}_k \left(\bar{w}_k + u_k \right) \right\} + \frac{\beta}{2} \cd_x \cdd_x \left( w_k - \bar{w}_k \right) \right]\\
    &= \Dt\left[ - \dfrac{\alpha}{6} \cd_x \left\{ \left( w_k - \bar{w}_k \right) \left( w_k + \bar{w}_k + u_k \right) \right\} + \frac{\beta}{2} \cd_x \cdd_x \left( w_k - \bar{w}_k \right) \right],
  \end{align*}
  which implies
  \begin{align*}
    \left| \phi_k(w) - \phi_k(\bar{w}) \right| &\leq \Dt\left[ \dfrac{|\alpha|}{6} \frac{1}{\Dx} \left\|  w - \bar{w} \right\|_\infty \left\| w + \bar{w} + u \right\|_\infty + \frac{|\beta|}{2} \frac{3}{(\Dx)^3} \left\| w - \bar{w} \right\|_\infty \right]\\
    &\leq \frac{\Dt}{(\Dx)^3} \left[ \dfrac{|\alpha|}{6} (\Dx)^2 (2 q + 1) r + \frac{3}{2} |\beta| \right] \left\| w - \bar{w} \right\|_\infty.
  \end{align*}
  Therefore, if
  \begin{align*}
    \Dt &< (\Dx)^3 \left[ \dfrac{|\alpha|}{6} (\Dx)^2 (2 q + 1) r + \frac{3}{2} |\beta| \right]^{-1} = \varepsilon_2( q, r),
  \end{align*}
  then
  \begin{align*}
    \left\| \phi(w) - \phi(\bar{w}) \right\|_\infty &\leq \frac{\Dt}{(\Dx)^3} \left[ \dfrac{|\alpha|}{6} (\Dx)^2 (2 q + 1) r + \frac{3}{2} |\beta| \right] < \left\| w - \bar{w} \right\|_\infty.
  \end{align*}

  As a conclusion, if $\Dt < \min\{ \varepsilon_1( q, r,\Dx), \varepsilon_2( q, r,\Dx) \}$, then $\phi$ is a contraction mapping on $B(q,r)$.

\section{Remaining proof of Lemma~\ref{lem:e:1}} \label{app:lem:e:1}
  The remaining proof of Lemma~\ref{lem:e:1} is as follows. Using the notation introduced in Subsection~\ref{subsec:proof_lemma} and leveraging skew-symmetry, we obtain
  \begin{align*}
    \fd_t \left\| \ed{m}{} \right\|^2 &= 2 \left\langle \fd_t \ed{m}{}, \fa_t \ed{m}{} \right\rangle\\
    &= - \dfrac{\alpha}{3} \left\langle \cd_x \bar{e}, \fa_t \ed{m}{} \right\rangle - \dfrac{\alpha}{3} \left\langle \cd_x f, \fa_t \ed{m}{} \right\rangle - 2\left\langle \td{m}{}, \fa_t \ed{m}{} \right\rangle\\
    &= \dfrac{\alpha}{3} \left\langle \bar{e}, \cd_x \fa_t \ed{m}{} \right\rangle - 2\left\langle \td{m}{}, \fa_t \ed{m}{} \right\rangle\\
    &\hspace{5mm} + \dfrac{\alpha}{3} \sum_{j=0}^1 \left\langle \ed{m+j}{} \left( 2\tud{m+j}{} + \tud{m+1-j}{} \right), \cd_x \fa_t \ed{m}{} \right\rangle\\
    &\leq \dfrac{|\alpha|}{6} \left( \left\| \bar{e} \right\|^2 + \left\| \cd_x \fa_t \ed{m}{} \right\|^2 \right) + \left\| \td{m}{} \right\|^2 + \left\| \fa_t \ed{m}{} \right\|^2 \\
    &\hspace{5mm} + \dfrac{|\alpha|}{6} \sum_{j=0}^1 \left( \left\| \ed{m+j}{} \right\|^2 \left\| 2\tud{m+j}{} + \tud{m+1-j}{} \right\|_\infty^2 + \left\| \cd_x \fa_t \ed{m}{} \right\|^2 \right)\\
    &\leq \dfrac{|\alpha|}{6} \left( 36 q^2 r^2 \fa_t \left\| \ed{m}{} \right\|^2 + \fa_t \left\| \fd_x \ed{m}{} \right\|^2 \right)\\
    &\hspace{5mm} + \dfrac{|\alpha|}{6} \left( 18 r^2 \fa_t \left\| \ed{m}{} \right\|^2 + 2 \fa_t \left\| \fd_x \ed{m}{} \right\|^2 \right) + \left\| \td{m}{} \right\|^2 + \fa_t \left\| \ed{m}{} \right\|^2\\
    &\leq C_1 \fa_t \left\| \ed{m}{} \right\|_{H^1}^2 + (c_0)^2 \left( (\Dt)^2+(\Dx)^2 \right)^2
  \end{align*}
  holds for ${\displaystyle C_1 \coloneqq \max\left\{ 3 |\alpha| ( 2 q^2 + 1) r^2 + 1, |\alpha|/2\right\}}$.



\end{document}